\documentclass[a4paper,11pt]{article}
\usepackage{amsmath}
\usepackage{amsthm}	 
\usepackage{amssymb}
\usepackage{verbatim}
\usepackage{subfigure}
\usepackage[section]{placeins}
\usepackage{texdraw}
\usepackage{color}
\usepackage{geometry}
\usepackage{graphicx}
\newtheorem {theorem}{Theorem }[section]

\newtheorem {lemma}{Lemma }[section]
\newtheorem {definition}{Definition}[section]
\newtheorem {remark}{Remark}[section]

\def\L{{\mathbf L}}

\title{ Existence of optimal boundary control for the Navier-Stokes equations with mixed boundary conditions}
\author{Telma Guerra \thanks{ESTB-IPS/ CMA, FCT-UNL, Portugal ({\tt telma.guerra@estbarreiro.ips.pt}). }, Ad\'elia Sequeira \thanks{CEMAT, IST-ULisboa ({\tt adelia.sequeira@math.ist.utl.pt})} and Jorge Tiago \thanks{CEMAT, IST-ULisboa ({\tt jftiago@math.ist.utl.pt}) This work has been partially supported  by FCT (Portugal) through the
Research Centers CMA/FCT/UNL, CEMAT-IST,  the grant SFRH/BPD/66638/2009 and the projects PEst-OE/MAT/UI0297/2014 and EXCL/MAT-NAN/0114/2012.}}

\date{}

\def\RR{{I~\hspace{-1.45ex}R} }
\def\NN{{I~\hspace{-1.45ex}N} }

\newcommand{\mbf}{\mathbf}

\DeclareMathOperator{\E}{E}  
\DeclareMathOperator{\sol}{S} 
\DeclareMathOperator{\pro}{P}  
\DeclareMathOperator{\C}{C}  

\begin{document}
\maketitle

\begin{abstract}
Variational approaches have been used successfully as a strategy  to take advantage from real data measurements. In several applications, this approach gives a means to increase the accuracy of numerical simulations. In the particular case of fluid dynamics, it leads to optimal control problems with non standard cost functionals which, when constraint to the Navier-Stokes equations, require a non-standard theoretical frame to ensure the existence  of solution. In this work, we prove the existence of solution for a class of such type of optimal control problems. Before doing that, we ensure the existence and uniqueness of solution for the 3D stationary Navier-Stokes equations, with mixed-boundary conditions, a particular type of boundary conditions very common in applications to biomedical problems.
\end{abstract}


\textbf{AMS} 49J20, 76D03, 76D05.

\pagestyle{myheadings}
\thispagestyle{plain}

\section{Introduction}
 Optimal control problems associated to fluid dynamics have been studied by several authors, during the last decades, motivated by the important applications of such type of problems to the industry. In a natural way, most of the first works were devoted to the case of distributed control as this is easier to handle. However, the most challenging problems in applications such as automobile or airplane design, and more recently, in bypass design or boundary reconstruction in medical applications,  are modeled by problems where  the  control is assumed to act on part of the boundary. Actually, boundary control problems are usually harder to deal, specially with respect to optimality conditions, since higher regularity for the solutions is often required. The list of works  on the subject is long, and  here we  only mention a few references  \cite{AT90}, \cite{GHS91}, \cite{FS92}, \cite{GM00}, \cite{DK05}, \cite{DT07} and \cite{DY09}.

In this work, and having in mind applications in biomedicine, we will consider the steady Navier-Stokes equations with mixed boundary conditions

\begin{equation}\label{navierstokes}
\left\{ \begin{array}{ll}-\nu\Delta {u}+
	{{u}}\cdot \nabla {{u}} +\nabla p= {f}& \qquad
	      \mbox{in }  \Omega,\vspace{2mm} \\
             \nabla \cdot{{u}}=0&  \qquad \mbox{in }  \Omega,\vspace{2mm}\\
             \gamma{u}={g}& \qquad \mbox{on }  \Gamma_{in},\\
            \gamma{{u}}={0}& \qquad \mbox{on }  \Gamma_{wall},\\
             \nu\partial_n{{u}}-pn={0}& \qquad \mbox{on }  \Gamma_{out},
\end{array}\right.\end{equation}
where $\nu$ represents the  viscosity of the fluid (possibly divided by its constant density),  $f$ the vector force acting on the fluid and $g$ the function imposing the velocity profile on $\Gamma_{in}$. The unknowns are the velocity vector field $u$ and the pressure variable $p$. These equations  have been widely used to model and simulate the blood flow in the cardiovascular system (see, for instance, \cite{G08} and the references cited therein). In this type of applications it is often required to represent part of an artery as the computational (bounded) domain $\Omega$. In addition, for the numerical simulations, we impose homogeneous Dirichlet boundary conditions on the surface representing the vessel wall ($\Gamma_{wall}$) and Dirichlet non-homogeneous on the artificial boundary ($\Gamma_{in}$), which is used to truncate the vessel from the upstream region. Besides, on the surface limiting the domain, in the downstream direction ($\Gamma_{out}$), homogeneous Neumann  boundary conditions are imposed. In Figure~\ref{figdomain} we can see a longitudinal section of such a domain, where the deformation of $\Gamma_{wall}$ could represent the presence of a plaque of atherosclerosis. 

\begin{figure}[h]
\centering
\includegraphics[scale=.45]{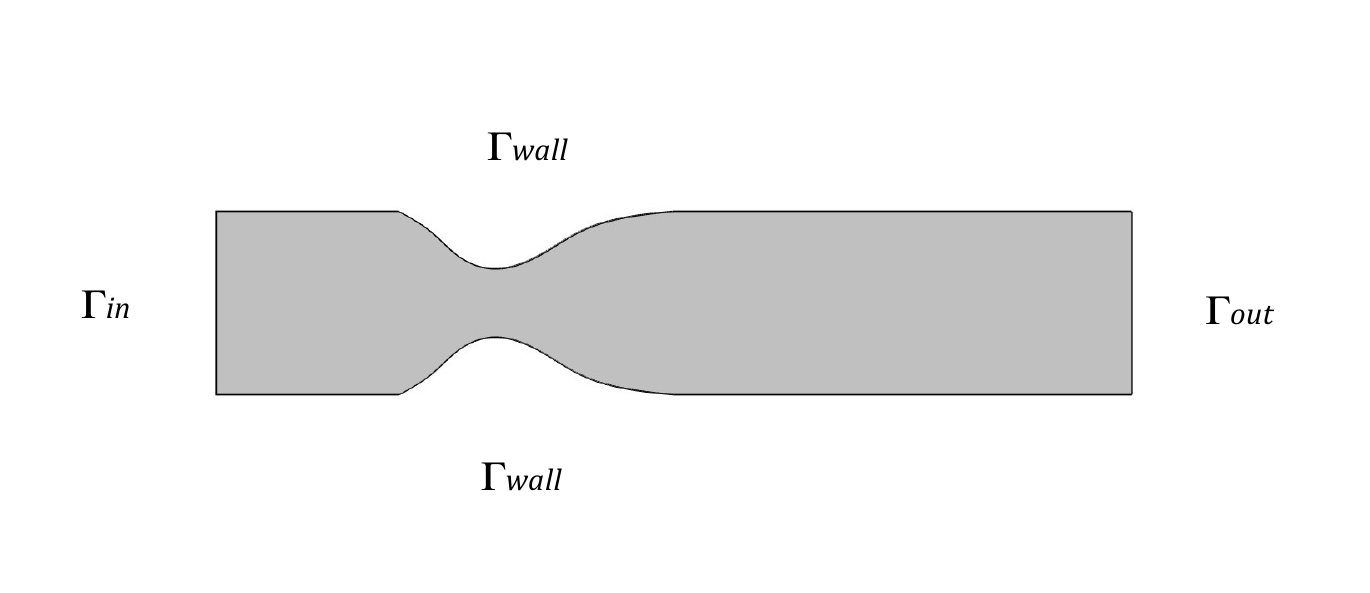}
 \centering
 \caption{\scriptsize{Representation of the domain $\Omega$}}
  \label{figdomain}
\end{figure}

When facing this and other type of pathologies of the cardiovascular system, it is important the evaluation of hemodynamical factors to predict, in a non invasive way, either the evolution of the disease, or the effect of possible therapies. This can be done by relying on the numerical simulations obtained in the domain under analysis. The main difficulty in this strategy lies in the lack of accuracy of the virtual simulations with respect to the real situation. In order to improve the accuracy and make the simulations sound enough, it is possible to use data from measurements of the blood velocity profile, obtained through medical imaging in some smaller parts of the vessel. This can be done through a variational approach, i.e., by setting an optimal control problem with a cost function (or a class of cost functions) of the type   
\begin{equation}\label{costfunctional}
J({u},{g})=\beta_1\int_{\Omega_{part}}|{{u}}-{{u}}_d|^2\,dx+\beta_2\int_{\Gamma_{in}}|{ g}|^2\,ds+\beta_3\int_{\Gamma_{in}}|{\nabla_s g}|^2\,ds,
\end{equation}
where ${{u}}_d$ represents the data available only on a part of the domain called $\Omega_{part}$. Note that, while fixing the weights $\beta_1$, $\beta_2$ and $\beta_3$, we determine whether the minimization of $J$ emphasizes more a good approximation of the velocity vector to $u_d$, a ``less expensive'' control $g$ (in terms of the $L^2$-norm), or a smoother control.
An example of $u_d$, measured in $\Omega_{part}$, could be the velocity vectors obtained in several cross sections of the vessel, as represented in Figure~\ref{figdata}.
\begin{figure}[h]
\centering
\includegraphics[scale=1]{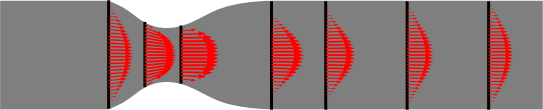}
 \centering
 \caption{\scriptsize{Representation of $u_d$ over $\Omega_{part}$}}
  \label{figdata}
\end{figure}

Solving the optimal control problem
\begin{equation}\label{controlproblem}
(P)\left\{
  \begin{array}{ll}
    \text{Minimize}\,\, J({u},{g})\\
    \text{subject to}\,\, \eqref{navierstokes}
  \end{array}
\right.
\end{equation}
will give us the means of making blood flow simulations more reliable, using known data. 

This strategy is not new, and has already been used as a proof of concept in \cite{GTS14} and \cite{TGS14}, where both the Navier-Stokes and the Generalized Navier-Stokes equations were considered to model the blood flow.
Even if it proved to be successful from the numerical point of view, problem $(P)$ has not yet been studied, at least up to the authors knowledge, not even with respect to the existence of solution. In fact, many authors have treated similar problems, considering the same type of cost functionals constrained to  the Navier-Stokes equations, but for the case where $\Omega_{part}=\Omega$ and without using mixed boundary conditions. In  \cite{DK05} and \cite{DY09} the case with only Dirichlet boundary conditions, and a similar cost functional, was treated. In \cite{GHS91} and \cite{M07} the authors considered $J$ as the cost functional, with $\Omega_{part}=\Omega$, but again they just dealt with Dirichlet boundary conditions. In \cite{FR09} the authors considered a more complex set of mixed boundary condition, but for a different cost functional.

Here we prove the existence of solution for problem $(P)$ regarded in the weak sense. We will make the distinction between different possibilities both for $\Omega_{part}$ and for the parameters $\beta_2$ and $\beta_3$. In order to do that, we will start by setting  the existence of a unique weak solution for the state equation (\ref{navierstokes}). The regularity of this solution remains an open problem and will not be treated here. It is important to deal with this issue, before addressing the natural following stages, namely the derivation of optimality conditions for problem $(P)$ and the numerical approximation. 

The organization of this paper reads as follows. In Section 2 we give some notation and results needed for this work. The Navier-Stokes equations with mixed boundary conditions are studied in Section 3. Finally, in Section 4, we prove the existence of solution for a class of optimal control problems.  

\section{Notation and some useful results}

We consider $\Omega\subset{R}^n$, with $n=2,3$, an open bounded subset with Lipschitz boundary.

The standard Sobolev spaces are denoted by 
$$ W^{k,p}(\Omega)=\left\{u\in \L^{p}(\Omega): \,\Vert u\Vert_{W^{k,p}}^p=\sum_{|\alpha|\leq k}\Vert D^{\alpha}u\Vert_{L^p}^p<\infty \right\},$$
 where $k\in \NN$ and $1<p<\infty$.  For $s\in\RR$, $W^{s,p}(\Omega)$ is defined by interpolation. The dual space of  $ W^{1,p}_0(\Omega)$ is denoted by  $ W^{-1,p'}(\Omega)$. We also use $H^s(\Omega)$ to represent the  Hilbert spaces $W^{s,2}(\Omega)$. For $\Gamma\subset\partial\Omega$ with positive measure we denote by $H^s(\Gamma)$, $s\geq\frac{1}{2}$, the image of the unique linear continuous trace operator  $$\gamma_{\Gamma}:H^{s+\frac{1}{2}}(\Omega)\to H^s (\Gamma),$$
such that $\gamma_{\Gamma}u=u_{|\Gamma}$ for all $u\in H^{s+\frac{1}{2}}(\Omega)\cap C^0(\bar{\Omega})$. In particular, for $s=0$, $H^0(\Gamma)$ is the subspace of $L^2(\Gamma)$ corresponding to the image of the continuous functions in $H^1(\Omega)$. The norm of $H^s(\Gamma)$ is defined similarly to the norm in $H^1(\Omega)$, except that the tangential derivatives on $\Gamma$ should be used (see, for instance, \cite{GHS91}).  
Whenever $Y$ is a space of functions $u:\Omega\to R$, we will use the boldface notation $ \mbf{Y}=Y\times Y\times Y$ for the corresponding space of vector valued functions.

We will also make use of the following Sobolev embedding result:
\begin{lemma}\label{sobolev}
  Let  $\Omega$ be a bounded set of class $C^1$. Assume that $p<n$ and $p^*=\frac{pn}{n-p}$. Then
\begin{description}
\item[i)] $W^{1,p}(\Omega)\subset L^q$, $\forall q\in [1,p^*[$ with compact embedding.
\item[ii)] $W^{1,p}(\Omega)\subset L^{p*}$, with continuous embedding.
\end{description}
\end{lemma}
\begin{proof}
For the proof see, for instance,  \cite{B83}, Corollary IX.14 and Theorem IX.16 - Remark 14ii).
\end{proof}
We consider the spaces of divergence free functions defined by
	$${ H}=
        \left\{ u\in {H}^1(\Omega)
        \mid \nabla\cdot u=0\right\},$$       
	$${ V}_{wall}=
        \left\{\psi\in  H_{\Gamma_{wall}}(\Omega)
        \mid \nabla\cdot \psi=0\right\}$$
and
	$${ V}_D=
        \left\{\psi\in  H_{\Gamma_D}(\Omega)
        \mid \nabla\cdot \psi=0\right\},$$
where $\Gamma_D$ refers to the Dirichlet boundary   $\Gamma_{in}\cup\Gamma_{wall}$.
In these definitions, for $\Gamma\in\{\Gamma_{wall},\Gamma_D\}$, we represent by ${H}_{\Gamma}$ the set $${H}_{\Gamma}=\left\{\psi\in  H^1(\Omega) \mid \gamma_{\Gamma} \psi=0\right\}.$$ 
The corresponding norms are defined by
$$\Vert.\Vert_{H}=\Vert.\Vert_{V_{D}}=\Vert.\Vert_{V_{wall}}=\Vert.\Vert_{H^1(\Omega)}.$$ 
 
        We also define 
$${H}_{0}^1(\Gamma)= \left\{v\in  L^2(\Gamma)
        \mid \nabla_s v \in  L^2(\Gamma), \, \gamma_{\partial\Gamma} v=0 \right\}$$ and
$${H}_{00}^{\frac{1}{2}}(\Gamma)= \left\{g\in  L^2(\Gamma)
        \mid \exists v \in  H^{1}(\Omega), \,v_{|_{\partial\Omega}} \in  H^{\frac{1}{2}}(\partial\Omega), \, \gamma_{\Gamma} v=g,\, \gamma_{\partial\Omega\setminus\Gamma} v=0 \right\}$$ a closed subspace of $H^{\frac{1}{2}}(\Gamma)$.

Note that we have the continuous embeddings $H_0^1(\Gamma) \subset H_{00}^{\frac{1}{2}}(\Gamma)$ and $ H_{00}^{\frac{1}{2}}(\Gamma)\subset L^2(\Gamma)$ (\cite{DL00}, pp. 397).

Finally, we set 
$$\hat{H}^{\frac{1}{2}}(\Gamma_1\cup\Gamma_2)=\left\{(g_1,g_2)\in H_{00}^{\frac{1}{2}}(\Gamma_1)\times H_{00}^{\frac{1}{2}}(\Gamma_2)\mid \int_{\Gamma_1}g_1\cdot n\, ds+\int_{\Gamma_2}g_2\cdot n\, ds=0   \right\}.$$
\section{State Equation}
The well-posedness of system (\ref{navierstokes}) concerning the existence and uniqueness for $g$ within an admissible class is required before studying the existence of solution of the optimal control problem. In \cite{KS98} the authors studied the evolutionary case setting the existence of a solution local in time, for the type of boundary conditions considered here. Concerning the stationary case, in \cite{K98} and \cite{G08} the existence of solution for a similar system was proved. Both authors considered Neumann conditions mixed with Dirichlet homogeneous conditions. In the later it was mentioned that no additional difficulties should be expected with non-homogeneous boundary conditions. In \cite{FR09}, the existence was shown, in the 2D case, for a system with mixed boundary conditions including Dirichlet non-homogeneous. Again the authors mentioned that the 3D case  could be proved using the same techniques. For the sake of clearness,  we show that system (\ref{navierstokes}) is in fact well-posed in the 3D case, following the ideas of \cite{FR09}.

We first start by considering the Stokes system
\begin{equation}\label{stokes}
\left\{ \begin{array}{ll}-\nu\Delta {u}+\nabla p= {h}& \qquad
	     \ \mbox{in} \ \Omega,\vspace{2mm} \\
             \nabla \cdot{{u}}=0&  \qquad \mbox{in} \ \Omega,\vspace{2mm}\\
             \gamma{u}={g}& \qquad \mbox{on} \ \Gamma_{in},\\
            \gamma{{u}}={0}& \qquad \mbox{on} \ \Gamma_{wall},\\
            \nu \partial_n{{u}}-pn={0}& \qquad \mbox{on} \ \Gamma_{out},
\end{array}\right.\end{equation}

\begin{definition}
Let ${g}\in\mbf{H}_{0}^1(\Gamma_{in}) $, ${h}\in \mbf{L}^{\frac{3}{2}}(\Omega)$. We call $u\in\mbf{V}_{{wall}}$ a weak solution of (\ref{stokes}) if $\gamma_{\Gamma_{in}}u=g$ and  
\begin{equation} \label{weakstokes}
\quad\nu\int\limits_{\Omega}\nabla u:\nabla v\,dx = \int\limits_{\Omega}{h} v\,dx,
             \end{equation}
for all $v\in\mbf{V}_D$.
\end{definition}
\begin{theorem}.
\begin{description}
\item [i)] There exists a unique solution $ u\in \mbf{V}_{wall}$ of problem $(\ref{weakstokes})$. For such solution there exists a distribution $p\in\mbf{L}^{\frac{3}{2}}(\Omega)$ such that $( u,p)\in{V}_{wall}\times{L}^2(\Omega)$ is a solution of (\ref{stokes}) in the sense of distributions. If $u$ and $p$ are smooth enough, then $p$ is unique and the boundary conditions in (\ref{stokes}) are verified point-wise.
\item[ii)] On the other hand, if $( u,p) \in{H}_{\Gamma_{wall}}\times L^{\frac{3}{2}}(\Omega)$ is a solution of problem $(\ref{stokes})$  in the sense of distributions, then $ u$ is a solution of (\ref{weakstokes}).
\end{description}
\end{theorem}
 
%
\begin{proof}

\begin{description}
\item[i)] Consider the auxiliar minimization problem 
$$\min_A E(u):=\frac{1}{2}\Vert \nabla v\Vert_{\mbf{L}^2(\Omega)}^2-(h,u)$$
where 
$$A=\{u\in \mbf{H}_{\Gamma_{wall}}, \, \gamma_{\Gamma_{in}}u=g\}.$$
The functional $E:\mbf{H}^1(\Omega)\to \mathbb{R}$ is continuous and convex on  $\mbf{H}^1(\Omega)$ and thus weakly lower semi-continuous with respect to the $\mbf{H}^1(\Omega)$ norm. Also, the admissibility set $A$ is sequentially weakly closed. 
Finally, since $E$ verifies the coercivity property, the classical theory  of the calculus of variations ensures the existence of a unique solution $\bar{u}$ for the minimization problem. Hence, $\bar{u}$ is also the unique solution of the necessary and sufficient optimality condition  
$$\nu\int\limits_{\Omega}\nabla u:\nabla v\,dx = \int\limits_{\Omega}{h} v\,dx,\quad \forall v\in\mbf{H}_{\Gamma_D}$$ and therefore (\ref{weakstokes}) has a unique solution.

If we take $v\in  \mbf{H}_{\Gamma_D}\cup\mbf{C}_0^{\infty}(\Omega)$ and integrate (\ref{weakstokes}) by parts,  we obtain 
$$\int_{\Omega}(\nu\Delta \bar{u} +h)\cdot v=0\Leftrightarrow (\Delta \bar{u} +h, v)=0, \quad \forall v\in \mbf{H}_{\Gamma_D}\cup\mbf{C}_0^{\infty}(\Omega).$$ 
Due to the  inclusion $L^{\frac{3}{2}}(\Omega)=(L^{{3}}(\Omega))'\subset (W_0^{1,{3}}(\Omega))'=W^{-1,\frac{3}{2}}(\Omega) $, we have  $\nu\Delta \bar{u} +h\in \mbf{W}^{-1,\frac{3}{2}}(\Omega)$. Therefore by De Rham's theorem (\cite{S01} Lemma II.2.2.2) there exits a distribution $p\in L^{\frac{3}{2}}(\Omega)$  such that $\nabla p\in \mbf{L}^{\frac{3}{2}}(\Omega)$ and $(\nu\Delta \bar{u} +h, v)=(\nabla p,v)$ that is, system (\ref{stokes}) is verified in the sense of distributions. 
Let us now assume that $\bar{u}$ and $p$ are smooth and replace $h$ by $-\nu\Delta \bar{u} + \nabla p$ in (\ref{stokes}). Integrating by parts we obtain
$$\int_{\Gamma_{out}}(\nu \partial_n\bar{u}-pn) \cdot v \,ds=0\, ,\quad \forall v\in \mbf{V}_D.$$  
Now consider $w\in\mbf{C}_0^{\infty}(\Gamma_{out})$ such that $\int_{\Gamma_{out}} w\cdot n\, ds=0$. If we define 
\begin{equation}
\bar{w}=\left\{ \begin{array}{ll} w & \qquad \mbox{on }  \Gamma_D=\Gamma_{in}\cup\Gamma_{wall}\\
             0 & \qquad \mbox{on }  \Gamma_{out},
\end{array}\right.\end{equation}
we have
$\bar{w}\in\mbf{C}_0^{\infty}(\partial\Omega)$ and $\int_{\partial\Omega} \bar{w}\cdot n\, ds=0$. As a result, there exists $v\in\mbf{V}_D$ such that $\gamma_{\partial\Omega}v=\bar{w}$ and $\gamma_{\Gamma_{out}}v=w$. Consequently,
$$\int_{\Gamma_{out}}(\nu \partial_n\bar{u}-pn) \cdot w \,ds=0,\quad \forall w\in \mbf{C}_0^{\infty}(\Gamma_{out})\text{ such that } \int_{\Gamma_{out}} w\cdot n\, ds=0.$$  
In view of a corollary of the fundamental lemma of the calculus of variations (\cite{D09} Cor.1.25 p.23), we have 
$$\nu \partial_n\bar{u}-pn=c_0n \text{ on }\Gamma_{out}, $$ where $c_0$ is a constant. 
Let us now take $\bar{p}=p+c$ as another distribution such that (\ref{stokes}) is verified. Then we have   
$$0=\int_{\Gamma_{out}}(\nu \partial_n\bar{u}-pn) \cdot v=\int_{\Gamma_{out}}(c-c_0)n \cdot v \,ds\quad \forall v\in \mbf{V}_D.$$  
Choosing $v$ such that $\int_{\Gamma_{out}}n \cdot v \,ds=1$, we conclude that $(\bar{u},\bar{p})$, with $c=c_0$, is the unique solution of (\ref{stokes}).
\item[ii)] If $u \in\mbf{H}_{\Gamma_{wall}}$ is a solution of (\ref{stokes}) then it is clear that $u\in \mbf{V}_{wall}$ and, as a result of  integration by parts, that (\ref{weakstokes}) is verified.
\end{description}
\end{proof}

Before obtaining an estimate for the Stokes problem, we first  recall some related results. 
\begin{lemma}\label{generalextension}
Let $g\in\mbf{H}^{\frac{1}{2}}(\partial\Omega) $ be such that $$\int_{\partial\Omega\setminus\Gamma}g\cdot n\, ds=\int_{\Gamma}g\cdot n\, ds=0.$$ Then there exists $v\in\mbf{H}$ such that $\gamma v=g$. 
\end{lemma}
\begin{proof}See, for instance, \cite{GR86}.
\end{proof}
It is now straightforward to prove the next lemma.
\begin{lemma}\label{inoutextension}
Let $(g_1,g_2)\in \hat{H}^{\frac{1}{2}}(\Gamma_{in}\cup\Gamma_{out})$. Then there is a bounded extension operator $E:\hat{H}^{\frac{1}{2}}(\Gamma_{in}\cup\Gamma_{out})\to\mbf{V}_{wall}$, $\forall v\in\mbf{V}_{wall}$, such that for $v=E(g_1,g_2)$ we have $g_1=\gamma_{\Gamma_{in}}v, \, g_2=\gamma_{\Gamma_{out}}v$.
\end{lemma}



As a result, we can  obtain the following estimate for the solution.
\begin{lemma}\label{stokesestimate}
Let $\sol:\mbf{H}_{00}^{\frac{1}{2}}(\Gamma_{in})\times\mbf{L}^{\frac{3}{2}}(\Omega)\to\mbf{V}_{wall}$ be the solution operator to (\ref{weakstokes}). Then, if $v=\sol(g,h)$, we have 
$$\|v \|_{\mbf{V}_{wall}}^2=\|v \|_{\mbf{H}^1(\Omega)}^2\leq c \left(\|g\|_{\mbf{H}_{00}^{\frac{1}{2}}(\Gamma_{in})}^2+\|h\|_{\mbf{L}^{\frac{3}{2}}(\Omega)}^2\right),$$
where $c>0$ is independent of $(g,h)$.
\end{lemma}
\begin{proof}
Using Lemma~\ref{inoutextension} we see that $v=\E g+\bar{v}$ with $\bar{v}=v-\E g\in \mbf{V}_D$. Hence
$$\|\nabla v\|_{{\bf L}^2(\Omega)}^2=(\nabla v,\nabla \E g)+(\nabla v,\nabla \bar{v}),$$ which, in view of the definition of weak solution, can be written as 
$$\|\nabla v\|_{{\bf L}^2(\Omega)}^2=(\nabla v,\nabla \E g)+\frac{1}{\nu}(h, \bar{v}).$$
We deal with each term of the right-hand side separately. Using Young's inequality, together with the fact that $\E$ is bounded, we have
\begin{align}
|(\nabla v,\nabla \E g)|&\leq c_1\|\nabla v\|_{{\bf L}^2(\Omega)}\|\nabla \E g\|_{{\bf L}^2(\Omega)}\leq c_2\|\nabla v\|_{{\bf L}^2(\Omega)}\| \E g\|_{\mbf{H}^1(\Omega)}\\
&\leq c_3\|\nabla v\|_{{\bf L}^2(\Omega)}\| g\|_{\mbf{H}_{00}^{\frac{1}{2}}(\Gamma_{in})}\leq \varepsilon \|\nabla v\|_{{\bf L}^2(\Omega)}^2+\frac{c_4}{\varepsilon}\| g\|_{\mbf{H}_{00}^{\frac{1}{2}}(\Gamma_{in})}^2,
\end{align} for $\varepsilon >0$.
Moreover, using Poincar\'e and Young inequalities and the Sobolev embedding $\mbf{H}^1(\Omega)\subset \mbf{L}^3(\Omega)$ (see Lemma~\ref{sobolev}.i), we have 
\begin{align}
|( h, \bar{v})|&\leq c_5\|h\|_{{\bf L}^\frac{3}{2}(\Omega)}\|\nabla \bar{v}\|_{{\bf L}^2(\Omega)}\leq \varepsilon \|\nabla \bar{v}\|_{{\bf L}^2(\Omega)}^2+\frac{c_6}{\varepsilon}\| h\|_{{\bf L}^\frac{3}{2}(\Omega)}^2.
\end{align}
And, by similar arguments,
\begin{align}
\|\nabla \bar{v}\|_{{\bf L}^2(\Omega)}^2=\|\nabla v-\nabla \E g\|_{{\bf L}^2(\Omega)}^2&\leq c_7\left(\|\nabla v\|_{{\bf L}^2(\Omega)}^2+\| \E g\|_{\mbf{H}^1(\Omega)}^2\right)\\
&\leq c_8\left(\|\nabla v\|_{{\bf L}^2(\Omega)}^2+\| g\|_{\mbf{H}_{00}^{\frac{1}{2}}(\Gamma_{in})}^2\right).
\end{align}
Therefore 
$$\|\nabla v\|_{{\bf L}^2(\Omega)}^2\leq \varepsilon (1+c_8) \|\nabla {v}\|_{{\bf L}^2(\Omega)}^2+\frac{c_6}{\varepsilon}\| h\|_{{\bf L}^\frac{3}{2}(\Omega)}^2+(\frac{c_4}{\varepsilon}+c_8\varepsilon)\| g\|_{\mbf{H}_{00}^{\frac{1}{2}}(\Gamma_{in})}^2$$ and consequently 
$$\|v\|_{\mbf{H}^1(\Omega)}^2\leq c_9\|\nabla v\|_{{\bf L}^2(\Omega)}^2\leq c\left (\| h\|_{{\bf L}^\frac{3}{2}(\Omega)}^2+\| g\|_{\mbf{H}_{00}^{\frac{1}{2}}(\Gamma_{in})}^2\right)$$ for a certain constant $c>0$.

\end{proof}

We can now prove the existence of a solution for the Navier-Stokes system (\ref{navierstokes}).

\begin{definition}
Let ${g}\in\mbf{H}_{0}^1(\Gamma_{in}) $, ${f}\in \mbf{L}^{\frac{3}{2}}(\Omega)$. We say that $u\in\mbf{V}_{wall}$ is a weak solution of (\ref{navierstokes}) if $\gamma_{\Gamma_{in}}u=g$ and  
\begin{equation} \label{weaknavierstokes}
\quad\nu\int\limits_{\Omega}\nabla u:\nabla v\,dx+\int\limits_{\Omega}({u}\cdot\nabla) u v\,dx = \int\limits_{\Omega}{f} v\,dx,
             \end{equation}
for all $v\in\mbf{V}_D$.
\end{definition}
We need the following result.
\begin{lemma}\label{convective}
If $u\in \mbf{H}^1(\Omega)$, then $u\cdot \nabla u\in \mbf{L}^{\frac{3}{2}}(\Omega)$ and $\|u\cdot \nabla u\|_{\mbf{L}^{\frac{3}{2}}(\Omega)}\leq\|u\|_{\mbf{H}^{1}(\Omega)}^2$.
\end{lemma}
\begin{proof}
Using H\"older's inequality (\cite{B83}, IV.2, Remark 2.) and  the Sobolev embedding $H^1(\Omega)\subset L^6(\Omega)$ (see Lemma~\ref{sobolev}.ii)) we have 
$$\int_{\Omega}|u\cdot \nabla u|^{\frac{3}{2}}\,dx\leq\|u\|_{\mbf{L}^6(\Omega)}^{\frac{3}{2}}\|\nabla u\|_{\mbf{L}^2(\Omega)}^{\frac{3}{2}}\leq c\|u\|_{\mbf{H}^1(\Omega)}^{\frac{3}{2}} \|\nabla u\|_{\mbf{L}^2(\Omega)}^{\frac{3}{2}}\leq c\|u\|_{\mbf{H}^{1}(\Omega)}^3\leq \infty.$$ 

\end{proof}
\begin{theorem}\label{navierstokesexistence}
Let ${g}\in\mbf{H}_{0}^1(\Gamma_{in}) $ such that $\|g\|_{\mbf{H}_{0}^1(\Gamma_{in}) }\leq \rho$, for $\rho>0$ sufficiently small,  and ${f}\in \mbf{L}^{\frac{3}{2}}(\Omega)$. Then, there exists a unique  weak solution  $u\in\mbf{V}_{wall}$ of the Navier-Stokes system (\ref{navierstokes}) which verifies
\begin{equation}\label{navierstokesestimate}\|u\|_{\mbf{H}^1(\Omega)}^2 \leq \alpha \left(\|g\|_{\mbf{H}_{0}^1(\Gamma_{in}) }^2\right)+\|f\|_{\mbf{{L}}^{\frac{3}{2}}(\Omega)}^2,
\end{equation}
where $\alpha(s)=c(s^2+s)$.
\end{theorem}
Before proceeding to the proof of the theorem, let us introduce another definition.

\begin{definition}
We define the projection operator $\pro:\mbf{L}^{\frac{3}{2}}(\Omega)\to \mbf{\hat{L}}^{\frac{3}{2}}(\Omega)$ as the solution of the equation 
$$(\pro h,v)=(h,v),\quad \forall v\in \mbf{\hat{L}}^3(\Omega),$$ where
$$\mbf{\hat{L}}^p(\Omega)=\left\{v\in \mbf{L}^p(\Omega)\mid \nabla \cdot v=0,\, \gamma_{\Gamma_D} (v\cdot n)=0 \right\}.$$
\end{definition}
 
\begin{proof} [Proof of Theorem~\ref{navierstokesexistence}.]
We look for $h\in\mbf{\hat{L}}^{\frac{3}{2}}(\Omega)$ such that the corresponding solution to the Stokes system $u=\sol(g,h)$ is also a solution of (\ref{weaknavierstokes}). For this purpose we will use a fixed point argument.
If we replace such $u=\sol(g,h)$ in (\ref{weaknavierstokes}), we get
$$\nu(\nabla \sol ,\nabla v)+(\sol\cdot\nabla \sol,v)=(f,v)\quad \forall v\in \mbf{V}_D, $$ 
which, by definition of $\sol$, is equivalent to 
\begin{equation}\nonumber
(h,  v)+(\sol\cdot\nabla \sol,v)=(f,v)\quad \forall v\in \mbf{V}_D
\end{equation}
which is also equivalent to 
\begin{equation}\label{eqproof}
(h+\sol\cdot\nabla \sol-f,v)=0\quad \forall v\in \mbf{V}_D. 
\end{equation}
Using Lemma~\ref{convective} and the fact that $\mbf{V}_D$ is dense in $\mbf{\hat{L}}^3(\Omega)$,  we can see that, from  equation (\ref{eqproof}), we have
\begin{align}
(\pro (h+\sol\cdot\nabla \sol-f),v)&=0\quad \forall v\in \mbf{\hat{L}}^3(\Omega)\Leftrightarrow\nonumber\\
(h+\pro (\sol\cdot\nabla \sol-f),v)&=0\quad \forall v\in \mbf{\hat{L}}^3(\Omega)\Leftrightarrow\nonumber\\
-\pro (\sol\cdot\nabla \sol-f)&=h\, .
\end{align}
We should now prove that the operator $\C:\mbf{\hat{L}}^{\frac{3}{2}}(\Omega)\to\mbf{L}^{3}(\Omega)$ defined by 
$$\C(h)=-\pro (\sol(g,h)\cdot\nabla \sol(g,h)-f)$$ verifies the contraction property.

Let $h_1,\, h_2\in B_{\delta}$, where $B_{\delta}\subset\mbf{\hat{L}}^{\frac{3}{2}}(\Omega)$ is a given ball with respect to the $\mbf{\hat{L}}^{\frac{3}{2}}(\Omega)$ metrics. Then, using H\"older's inequality together with Poincar\'e's inequality, we get 
{\small \begin{align}
&\|\C(h_1)-\C(h_2)\|_{\mbf{\hat{L}}^{\frac{3}{2}}(\Omega)}=\nonumber\\
&\|\pro (\sol(g,h_1)\cdot\nabla \sol(g,h_1)- \sol(g,h_2)\cdot\nabla \sol(g,h_2))\|_{\mbf{\hat{L}}^{\frac{3}{2}}(\Omega)}=\nonumber\\
&\|\sol(g,h_1)\cdot\nabla \sol(g,h_1)- \sol(g,h_2)\cdot\nabla \sol(g,h_2)\|_{\mbf{{L}}^{\frac{3}{2}}(\Omega)}\leq\nonumber\\
&\|\sol(g,h_1)\cdot\nabla \sol(g,h_1)-\sol(g,h_2)\cdot\nabla \sol(g,h_1)\|_{\mbf{{L}}^{\frac{3}{2}}(\Omega)}\nonumber\\ 
&+\|\sol(g,h_2)\cdot\nabla \sol(g,h_1)- \sol(g,h_2)\cdot\nabla \sol(g,h_2)\|_{\mbf{{L}}^{\frac{3}{2}}(\Omega)}=\nonumber\\
&\|\sol(0,h_1-h_2)\cdot\nabla \sol(g,h_1)\|_{\mbf{{L}}^{\frac{3}{2}}(\Omega)}+\|\sol(g,h_2)\cdot\nabla \sol(0,h_1-h_2)\|_{\mbf{{L}}^{\frac{3}{2}}(\Omega)}\leq\nonumber\\
&\|\sol(0,h_1-h_2)\|_{\mbf{{L}}^6(\Omega)}\|\nabla \sol(g,h_1)\|_{\mbf{{L}}^2(\Omega)}+\|\sol(g,h_2)\|_{\mbf{{L}}^6(\Omega)}\|\nabla \sol(0,h_1-h_2)\|_{\mbf{{L}}^2(\Omega)}\leq\nonumber\\
&c_1(\|\sol(0,h_1-h_2)\|_{\mbf{{H}}^1(\Omega)}\|\nabla \sol(g,h_1)\|_{\mbf{{L}}^2(\Omega)}+\|\sol(g,h_2)\|_{\mbf{{H}}^1(\Omega)}\|\nabla \sol(0,h_1-h_2)\|_{\mbf{{L}}^2(\Omega)})\leq\nonumber\\
&c_2\|\nabla\sol(0,h_1-h_2)\|_{\mbf{{L}}^2(\Omega)}\left(\|\nabla \sol(g,h_1)\|_{\mbf{{L}}^2(\Omega)}+\|\sol(g,h_2)\|_{\mbf{{H}}^1(\Omega)}\right).\label{major}
\end{align}}
Using Lemma~\ref{stokesestimate} and the continuous embedding $H_0^1(\Gamma_{in})\subset {H}_{00}^{\frac{1}{2}}(\Gamma_{in})$, we can see that 
{\small\begin{align}
(\ref{major})\,&\leq c_3\left(\|h_1-h_2\|_{\mbf{{L}}^{\frac{3}{2}}(\Omega)}^2\right)^{\frac{1}{2}}\times\nonumber\\ &\left[\left(\|h_1\|_{\mbf{{L}}^{\frac{3}{2}}(\Omega)}^2+\|g\|_{\mbf{H}_{00}^{\frac{1}{2}}(\Gamma_{in})}^2 \right)^{\frac{1}{2}}+\left(\|h_2\|_{\mbf{{L}}^{\frac{3}{2}}(\Omega)}^2+\|g\|_{\mbf{H}_{00}^{\frac{1}{2}}(\Gamma_{in})}^2 \right)^{\frac{1}{2}}  \right] \nonumber\\
&\leq c_4\|h_1-h_2\|_{\mbf{{L}}^{\frac{3}{2}}(\Omega)}\left[\|h_1\|_{\mbf{{L}}^{\frac{3}{2}}(\Omega)}+\|h_2\|_{\mbf{{L}}^{\frac{3}{2}}(\Omega)}+\|g\|_{\mbf{H}_{0}^{{1}}(\Gamma_{in})}  \right] \nonumber\\
&\leq \bar{c}\|h_1-h_2\|_{\mbf{{L}}^{\frac{3}{2}}(\Omega)},\nonumber
\end{align}}
where $\bar{c}$ depends on $\|h_1\|_{\mbf{{L}}^{\frac{3}{2}}(\Omega)}$, $\|h_2\|_{\mbf{{L}}^{\frac{3}{2}}(\Omega)}$ and $\|g\|_{\mbf{H}_{0}^{{1}}(\Gamma_{in})}$. But since $h_1,\, h_2\in B_{\delta}$, we can choose $\delta$ and $\rho$ small enough so that $\bar{c}<1$. Therefore $\sol$ maps $B_{\delta}$ into itself and hence it has a fixed point $\bar{h}$. Since $\bar{c}$ is strictly smaller than $1$, it is easy to see that such fixed point  is unique. 
As for the estimate (\ref{navierstokesestimate}),  let us notice that the fixed point can be obtained as the limit of a sequence $(h_k)$ verifying 
$$h_1=\C(0),\,h_2=\C(h_1),\ldots,\, h_k=\C(h_{k-1}),...$$ Since we have $h_k=\sum_{i=1}^k(h_i-h_{i-1})=\sum_{i=1}^k[\C(h_{i-1})-\C(h_{i-2})]$
then, in virtue of Lemma~\ref{convective} and Lemma~\ref{stokesestimate}, we have
\begin{align}
\|\bar{h}\|_{\mbf{{L}}^{\frac{3}{2}}(\Omega)}=&\|\lim_{k\to\infty}h_k\|_{\mbf{{L}}^{\frac{3}{2}}(\Omega)}\leq\lim_{k\to\infty}\sum_{i=1}^k\|h_k-h_{k-1}\|_{\mbf{{L}}^{\frac{3}{2}}(\Omega)}\nonumber\\
\leq&\sum_{i=1}^{\infty}\bar{c}^{\ i-1}\|\C(0)\|_{\mbf{{L}}^{\frac{3}{2}}(\Omega)}= \frac{\bar{c}}{1-\bar{c}}\|\sol(g,0)\cdot \nabla \sol(g,0)-f\|_{\mbf{{L}}^{\frac{3}{2}}(\Omega)}\nonumber\\
\leq& c_5(\|\sol(g,0)\|_{\mbf{{H}}^1(\Omega)}^2+\|f\|_{\mbf{{L}}^{\frac{3}{2}}(\Omega)})\leq c_6(\|g\|_{\mbf{{H}}_{00}^{\frac{1}{2}}(\Gamma_{in})}^2+\|f\|_{\mbf{{L}}^{\frac{3}{2}}(\Omega)})\nonumber.\\
\end{align}
Consequently, the solution $u=\sol(\bar{h},g)$ of system (\ref{weaknavierstokes}) is bounded by
\begin{align}
\|u\|_{\mbf{{H}}^{1}(\Omega)}^2&=\|\sol(g,\bar{h})\|_{\mbf{{H}}^{1}(\Omega)}^2\leq c_6\left( \|g\|_{\mbf{{H}}_{00}^{\frac{1}{2}}(\Gamma_{in})}^2+\|\bar{h}\|_{\mbf{{L}}^{\frac{3}{2}}(\Omega)}^2\right)\nonumber.\\
&\leq c_7\left( \|g\|_{\mbf{{H}}_{00}^{\frac{1}{2}}(\Gamma_{in})}^2+\|g\|_{\mbf{{H}}_{00}^{\frac{1}{2}}(\Gamma_{in})}^4+\|f\|_{\mbf{{L}}^{\frac{3}{2}}(\Omega)}^2\right)\nonumber.\\
&\leq c_8\left( \|g\|_{\mbf{{H}}_{0}^{1}(\Gamma_{in})}^2+\|g\|_{\mbf{{H}}_{0}^{1}(\Gamma_{in})}^4+\|f\|_{\mbf{{L}}^{\frac{3}{2}}(\Omega)}^2\right)\nonumber.\\
&= \alpha\left( \|g\|_{\mbf{{H}}_{0}^{1}(\Gamma_{in})}^2\right)+\|f\|_{\mbf{{L}}^{\frac{3}{2}}(\Omega)}^2\nonumber.
\end{align}
\end{proof}
\begin{remark}\label{remarklessregularity}
In the proof of the previous theorem the fact that $g\in \mbf{H}_0^1(\Gamma_{in})$ is not essential, and we could alternatively suppose that $g\in\mbf{H}_{00}^{\frac{1}{2}}(\Gamma_{in})$ verifies $\|g\|_{\mbf{H}_{00}^{\frac{1}{2}}(\Gamma_{in}) }\leq \rho$. In this case the proof could follow in the same way, but we would get the estimate
\begin{equation}\label{estimatelessregularity}\|u\|_{\mbf{H}^1(\Omega)}^2 \leq \alpha \left(\|g\|_{\mbf{H}_{00}^{\frac{1}{2}}(\Gamma_{in})}^2\right)+\|f\|_{\mbf{{L}}^{\frac{3}{2}}(\Omega)}^2,
\end{equation}
instead of (\ref{navierstokesestimate}).
\end{remark}
\section{Existence results for the optimal control problem}
Consider the admissible control set 
$${\cal U}=\left\{g\in  H_0^1(\Gamma_{in})
        \mid \,\|g\|_{H_0^1(\Gamma)}\leq \rho\right\},$$ where $\rho$ is defined as in Theorem~\ref{navierstokesexistence}.
We can define the weak version of problem $(P)$ as follows: we look for $g\in \cal U$ such that $J(u,g)$ is minimized, where $u$ is the unique weak solution of (\ref{weaknavierstokes}) corresponding to $g$.
\begin{remark}
Note that ${\cal U}$ is just an example of an admissible set, within the abstract set
$${\cal U}_0=\left\{g\in  H_0^1(\Gamma_{in}):\text{such that (\ref{weaknavierstokes}) has a unique solution}\right\}.$$
\end{remark}
We can prove the following existence result:
\begin{theorem}\label{existencecontrolfull}
Assume that $\Omega_{part}=\Omega$, $\rho$ is as described above and $\beta_2,\,\beta_3\neq 0$. Then $(P)$ has an optimal solution $(u,g)\in \mbf{V}_{wall}\times \cal{U}$ in the weak sense.
\end{theorem}
\begin{proof}
First see that for  $g=0$  there is a corresponding unique solution $u_0$  to (\ref{weaknavierstokes}) so that $\mbf{V}_{wall}\times \cal{U}$ is nonempty. This implies  that $0\leq J\leq +\infty$. 

Let $( u_k,g_k)_k\subset \mbf{V}_{wall}\times \cal{U} $ be a minimizing sequence, that is, such that  
$$J(u_k,g_k)\rightarrow I,\text{ the infimum, when }k\to+\infty.$$ Since ${\cal U}\subset H_0^1(\Gamma_{in})$ is bounded, there exists a subsequence of $(g_k)_k$ which converges weakly to a certain $\bar{g}\in H_0^1(\Gamma_{in})$. Due to (\ref{navierstokesestimate}) we have 
 $$\|u_k\|_{\mbf{H}^1(\Omega)}^2 \leq \alpha \left(\|g_k\|_{\mbf{H}_{0}^1(\Gamma_{in}) }^2\right)+\|f\|_{\mbf{{L}}^{\frac{3}{2}}(\Omega)}^2,\quad\forall k,$$
and therefore there exists $\bar{u}$ such that $u_k\rightarrow \bar{u}$ weakly in $\mbf{H}^1(\Omega)$. Indeed, we have $\bar{u}\in \mbf{V}_{wall}$, as both  the divergence operator and the trace operator $\gamma_{\Gamma_{wall}}:H^1(\Omega)\to H^{\frac{1}{2}}(\Gamma_{wall})$ are bounded linear operators. Also, as $\gamma_{\Gamma_{in}}u_k\rightarrow \gamma_{\Gamma_{in}}\bar{u}$, weakly in $\mbf{H}^{\frac{1}{2}}(\Gamma_{in})$, we have that  $\gamma_{\Gamma_{in}}u_k=g_k$ converges weakly in $\mbf{L}^2(\Gamma_{in})$, both to $\gamma_{\Gamma_{in}}\bar{u }$ and $\bar{g}$. Thus,  we must have $\gamma_{\Gamma_{in}}\bar{u}=\bar{g}$. Finally, since the convective term in (\ref{weaknavierstokes}) is weakly continuous in $\mbf{H}^1(\Omega)$ (see \cite{GR86} p.286) we conclude that $\bar{u}$ is the solution corresponding to $\bar{g}$. Due to the fact that the functional $J$ is both convex and continuous, and therefore strong lower semi-continuous (l.s.c.), it is also l.s.c. with respect to the weak topology (\cite{B83} Remark III.8.6). Consequently, 

$$I=\lim_{k} J(u_k,g_k)\geq\liminf_{k}J( u_k, g_k)\geq J(\bar{u},\bar{g})\geq I,$$
and we conclude that  $( \bar{u}, \bar{g})$ is a an optimal solution for $(P)$.
\end{proof}
\begin{remark}
The fact that we assume ${\cal U_0}$ bounded in $ H_0^1(\Gamma_{in})$ is a very strong assumption which allows us to prove the result even either if $\beta_2=0$ or $\beta_3=0$. In this latter case, the l.s.c. property of $J$ should be verified with respect to $H^{\frac{1}{2}}(\Gamma_{in})$ rather than $ H_0^1(\Gamma_{in})$. 
\end{remark}

\begin{remark} We can also choose an admissible set for the controls that is not necessarily bounded. This is the case when  ${\cal U}={\cal U}_0$.
Then, if $\beta_3\neq 0$, from the fact that for a minimizing sequence $(g_k)_k$ we have $$\|g_k\|_{H_0^1(\Gamma_{in})}\leq J(u_k,g_k)\leq +\infty,$$ we can still extract a weakly convergent sequence in   $H_0^1(\Gamma_{in})$, so that the proof would follow as above. If $\beta_3=0$, in view of  the properties of $ H_0^1(\Gamma_{in})$ (see for instance \cite{GHS91}), we would get $$\|g_k\|_{H_0^1(\Gamma_{in})}\leq\|g_k\|_{L^2(\Gamma_{in})}\leq J(u_k,g_k)\leq +\infty,$$
and the proof could be attained similarly as above.

\end{remark}
We will now consider another choice for $\Omega_{part}$ more connected to the medical applications we have in mind. Let $\Omega$ be a domain representing a blood vessel like in Figure~\ref{figdomain}. Consider $(\Omega_{p_i})_i$ to be a monotone sequence of subsets of $\Omega$, such that 
\begin{equation}\label{subsets}\Omega_{p_1}\subset\Omega_{p_2}...\subset\Omega_{p_m}\subset \Omega.
\end{equation}
In addition, assume also that for all $i\in\{1,...,m\}$, we have $$\partial\Omega_{p_i}=\Gamma_{in_i}\cup\Gamma_{wall_i}\cup\Gamma_{out_i}$$
where $\Gamma_{out_i}$, $i\in\{1,...,m\}$,  are disjoint surfaces corresponding to cross sections of $\Omega$, 
$\Gamma_{in_i}=\Gamma_{in}$,{ and } $\Gamma_{wall_i}=\Gamma_{wall}\cap\overline{\Omega}_{p_i}\neq\emptyset.$   
Note that the construction of each $\Omega_{p_i}$ in this way  ensures that (\ref{subsets}) is verified, and that  each $\Omega_{p_i}$ itself represents a part of  the vessel $\Omega$.

Now consider $\Omega_{part}=\cup_{i=1}^ms_i$ where $s_i=\Gamma_{out_i}$, for all $i\in\{1,...,m\}$. An example of such a situation is represented in Figure~\ref{figdata}. We can still establish the existence of solution in this case.

\begin{theorem}\label{existencecontrolsections}
Assume that $\Omega_{part}$ in $J$ is given by $\Omega_{part}=\cup_{i=1}^m{s_i}$, as described above. Then there is an optimal solution to problem $(P)$.
\end{theorem}
\begin{proof}
Let $\gamma_{s_i}:\mbf{H}^1(\Omega_{p_i})\to\mbf{H}^{\frac{1}{2}}(s_{i})$ be the family of linear, and bounded, trace operators defining the boundary values, over each surface $s_i$, for functions defined in $\Omega_{p_i}$. 
To prove that $J$ is weakly l.s.c, we need to see that it verifies the continuity and convexity properties. Let $u_k\rightarrow u$ in $\mbf{H}^1(\Omega)$ and consider $\gamma_{s_i}u_d=g_i$ to be the values of the known data over each $s_i$. 
In this case $$\left|\int_{\Omega_{part}}(u_k-u_d)^2-(u-u_d)^2\,ds\right|$$ is, in fact,
\begin{align}
&\left|\sum_{i=1}^m\left[\|\gamma_{s_i} u_k-g_i\|_{L^2(s_i)}^2 - \|\gamma_{s_i}u-g_i\|_{L^2(s_i)}^2\right]\right|\leq\nonumber\\
&\left|\sum_{i=1}^m\left[(\|\gamma_{s_i} u_k-\gamma_{s_i}u\|_{L^2(s_i)}+\|\gamma_{s_i}u-g_i\|_{L^2(s_i)}))^2 - \|\gamma_{s_i}u-g_i\|_{L^2(s_i)}^2\right]\right|.\nonumber
\end{align}
Due to the boundness of each $\gamma_{s_i}$ we have that the last term can be bounded from above by
\begin{align}
&\left|\sum_{i=1}^m\left[(c_i\| u_k-u\|_{\mbf{H}^1(\Omega_{p_i})}+\|\gamma_{s_i}u-g_i\|_{L^2(s_i)}))^2 - \|\gamma_{s_i}u-g_i\|_{L^2(s_i)}^2\right]\right|\leq\nonumber\\
&\left|\sum_{i=1}^m\left[(c_i\| u_k-u\|_{\mbf{H}^1(\Omega)}+\|\gamma_{s_i}u-g_i\|_{L^2(s_i)}))^2 - \|\gamma_{s_i}u-g_i\|_{L^2(s_i)}^2\right]\right|,\nonumber\\
\end{align}
which goes to zero when $k\rightarrow\infty$.

The convexity follows directly from the fact that 
\begin{align}
&\int_{\Omega_{part}}(\frac{u_1+u_2}{2}-u_d)^2\,ds=\sum_{i=1}^m\frac{1}{4}\int_{s_{i}}(\gamma_{s_i}u_1-g_i+\gamma_{s_i}u_2-g_i)^2\,ds\nonumber\\
&\leq\sum_{i=1}^m\frac{1}{4}\int_{s_{i}}2^1[(\gamma_{s_i}u_1-g_i)^2+(\gamma_{s_i}u_2-g_i)^2]\,ds\nonumber\\
&\leq\frac{1}{2}\int_{\Omega_{part}}(u_1-u_d)^2\,ds+\frac{1}{2}\int_{\Omega_{part}}(u_2-u_d)^2\,ds\nonumber.
\end{align}
Therefore $J$ is weakly l.s.c.. The rest of the proof follows as in Theorem~\ref{existencecontrolfull}.
\end{proof}
Lastly, another case that can also be interesting from the applications point of view.

\begin{theorem}\label{existencecontrolpartial}If we consider now $\Omega_{p_i}$ as a family of disjoint subdomains of $\Omega$ and we take $\Omega_{part}=\cup_{i=1}^m\Omega_{p_i}$ in $J$,  then problem $(P)$ also has an optimal solution.
\end{theorem}
\begin{proof}
To prove this statement, we will check, once more, that  $J$ remains convex and strongly continuous.  Concerning the convexity, it follows directly as in Theorem~\ref{existencecontrolsections}. As for the continuity, let $(u_k)_k$ be a convergent sequence to $u$ in $\mbf{H}^1(\Omega)$, then 
\begin{align}
&\left|\int_{\Omega_{part}}(u_k-u_d)^2-(u-u_d)^2\,dx\right|\leq \nonumber\\
&\left|\sum_{i=1}^m\left[(\|u_k-u\|_{L^2(\Omega_{p_i})}+\|u-u_d\|_{L^2(\Omega_{p_i})})^2 - \|u-u_d\|_{L^2(\Omega_{p_i})}^2\right]\right|\leq\nonumber\\
&\left|\sum_{i=1}^m\left[(\|u_k-u\|_{L^2(\Omega)}+\|u-u_d\|_{L^2(\Omega_{p_i})})^2 - \|u-u_d\|_{L^2(\Omega_{p_i})}^2\right]\right|\nonumber
\end{align}
which tends to zero when $k\rightarrow\infty$.
\end{proof}



\begin{thebibliography}{99}
\bibitem{AT90} F. Abergel and R. Temam, On some control problems in fluid mechanics. \emph{Theoret. Comput.
Fluid Dynamics} \textbf{1} (1990), 303--325.

\bibitem{B83} H. Br\'ezis, \emph{Analyse fonctionnelle, th\'eorie et applications}, Masson, Paris, 1983.

\bibitem{D09} B. Dacorogna, \emph{Introduction to the Calculus of Variations}, Imperial College Press, London, 2  Ed., 2009.

\bibitem{DL00} R. Dautray, J. L. Lions, \emph{Mathematical Analysis and Numerical Methods for Science and Technology}, vol. 2, Springer, Berlin, 2000.

\bibitem{DK05} J. C. De Los Reyes, K. Kunisch, A semi-smooth Newton method for control constrained
boundary optimal control of the Navier-Stokes equations. \emph{Nonlinear Anal.} \textbf{62} (2005), No. 7, 1289--1316.

\bibitem{DT07} J. C. De los Reyes, F. Troltzsch, Optimal control of the stationary Navier-Stokes equations with mixed control-state constraints. \emph{SIAM J. Control Optim.} \textbf{46} (2007), 604--629.

\bibitem{DY09} J. C. De Los Reyes, I. Yousept, Regularized state-constrained boundary optimal control of the Navier-Stokes equations. \emph{J. Math. Anal. Appl.} \textbf{356} (2009), 257--279.

\bibitem{FS92} H. Fattorini and S. Sritharan, Existence of optimal controls for viscous flow problems. \emph{Proc. Roy. Soc. London Ser. A} \textbf{439} (1992),  81--10.

\bibitem{FR09} A. V. Fursikov and R. Rannacher, Optimal Neumann control for the 2D steady state Navier-Stokes equations. In \emph{New Directions in Mathematical Fluid Mechanics} (ed. A. V. Fursikov, et al.), Advances in Mathematical Fluid Mechanics, Birkh\"auser, Basel 2009, 193--221.

\bibitem{G08} G. P. Galdi, Mathematical problems in classical and non-Newtonian fluid mechanics. In \emph{Hemodynamical Flows: Modeling, Analysis and Simulation} (ed. G. P. Galdi, A. M. Robertson, R. Rannacher and S. Turek ),  Oberwolfach Seminars,  Vol. 37, Birkh\"auser-Verlag, Basel 2008, 121--273.

\bibitem{GR86} V. Girault, P. A. Raviart, \emph{Finite Element Methods for the Navier-Stokes Equations}, Springer, Berlin, 1986.

\bibitem{GTS14} T. Guerra, J. Tiago, A. Sequeira, Optimal control in blood flow simulations. \emph{International
Journal of Non-Linear Mechanics} \textbf{64} (2014), 57--59.

\bibitem{GM00} M. Gunzburger and S. Manservisi, The velocity tracking problem for Navier-Stokes flows with boundary control. \emph{SIAM J. Contr. Optim.} \textbf{39} (2000), No. 2, 594--634.

\bibitem{GHS91} M. Gunzburger, L. Hou and T. Svobodny, Analysis and finite element approximation
of optimal control problems for the stationary Navier-Stokes equations with Dirichlet controls.
\emph{Mod\'el. Math. Anal. Num.} \textbf{25} (1991), 711--748.

\bibitem{K98} P. Kucera, Solutions of the Navier-Stokes equations with mixed boundary conditions in a bounded domain. In  \emph{Analysis, Numerics and Applications of Differential and Integral Equations} (ed. M. Bach, C. Constanda, G. C. Hsiao, A. M. Sandig and P. Werner), Pitman Research Notes in Mathematics, Series 379,  Addison Wesley, London, 1998, 127--131.

\bibitem{KS98} P. Kucera and Z. Skalak, Local Solutions to the Navier-Stokes Equations with 
Mixed Boundary Conditions. \emph{Acta Appl. Math.} \textbf{54} (1998), 275--288.

\bibitem{M07} S. Manservisi, An extended domain method for optimal boundary control for Navier-Stokes equations. \emph{Int. J. Numer. Anal. Mod.} \textbf{4} (2007), No. 3-4, 584--607.

\bibitem{S01} H. Sohr, \emph{The Navier-Stokes equations, An elementary functional analytic approach}. Birkh\"auser Advanced Texts: Baslera Lehrbucher, Birkh\"auser Verlag, Basel, 2001.

\bibitem{TGS14} J. Tiago, A. Gambaruto, A. Sequeira, Patient-specific blood flow simulations: setting Dirichlet
boundary conditions for minimal error with respect to measured data. \emph{Mathematical Models of
Natural Phenomena} \textbf{9} (2014), Iss. 6, 98--116.


\end{thebibliography}
\end{document}